\documentclass[12pt,reqno]{amsart}
\usepackage{amssymb}
\usepackage{a4wide}
\usepackage{graphicx,xcolor,enumerate}

\newtheorem{theorem}{Theorem}[section]
\newtheorem{lemma}[theorem]{Lemma}
\newtheorem{corollary}[theorem]{Corollary}

\newtheorem{problem}[theorem]{Problem}
\newtheorem{proposition}[theorem]{Proposition}

\newtheorem{claim}{Claim}[theorem]

\theoremstyle{definition}
\newtheorem{definition}[theorem]{Definition}

\theoremstyle{remark}

\newcommand{\itemprefix}{}
\makeatletter
\newcommand{\myitem}{%
\item\protected@edef\@currentlabel{\itemprefix\theenumi}%
}
\makeatother

\renewcommand{\>}{\right\rangle}

\makeatletter
\def\int{\mathop{\operator@font int}\nolimits}
\makeatother

\author[I. Juh\'asz]{Istv\'an Juh\'asz}
\address      {HUN-REN Alfr\'ed Rényi Institute of Mathematics}
\email{juhasz@renyi.hu}

\author[J. van Mill]{Jan van Mill}
\address{University of Amsterdam}
\email{j.vanMill@uva.nl}

\makeatletter
  \@namedef{subjclassname@2020}{%
  \textup{2020} Mathematics Subject Classification}
   \makeatother


\title[$\sigma$-compact spaces]{Countable tightness is not discretely reflexive in $\sigma$-compact spaces}

\begin{document}

\begin{abstract}
Answering a question raised by V. V. Tkachuk in \cite{VTk}, we present several examples of $\sigma$-compact spaces,
some only consistent and some in ZFC, that are not countably tight but in which the closure of any
discrete subset is countably tight. In fact, in some of our examples the closures of all discrete subsets
are even first countable.
\end{abstract}

\maketitle

\section{Introduction}

A topological property $\mathcal{P}$ is said to be {\em discretely reflexive} in a class $\mathcal{C}$ of spaces if any $X \in \mathcal{C}$
has property $\mathcal{P}$, whenever the closures of all discrete subsets of $X$ have property $\mathcal{P}$.
This concept had been introduced in \cite{ATW} where many interesting instances of it were established.
For instance, it was proved there that both countable character and countable tightness are discretely reflexive in the class
of compact $T_2$-spaces.

Much more recently V. V. Tkachuk in  \cite{VTk} raised the question if countable tightness is discretely reflexive in the more general class
of $\sigma$-compact $T_3$-spaces. In this note we are going to present several counterexamples to this. In fact, in some of our examples of
not countably tight $\sigma$-compact $T_3$-spaces the closures of discrete subsets are even first countable. These examples, of course,
show that first countability also fails to be discretely reflexive in the class of $\sigma$-compact $T_3$-spaces. This has been known before,
by considering countable $T_3$-spaces of uncountable character in which all discrete subsets are closed. Our examples, being very different from these,
could be of some interest in this respect too.

Our set-theoretic and topological notation and terminology are standard. Those concerning cardinal functions follow \cite{J}.

\section{Consistent examples}

We start by describing an operator that assigns to certain CCC 0-dimensional $T_2$-spaces $X$ a one-point extensions of $\omega \times X$
that are also 0-dimensional and $T_2$.

\begin{definition}\label{df:L(X)}
Let $X$ be a CCC 0-dimensional $T_2$-space such that $X = \bigcup \{K_\alpha : \alpha < \omega_1\}$, where each $K_\alpha$ is a non-empty closed
nowhere dense set in $X$.

For each $\alpha < \omega_1$ fix a maximal disjoint collection $\{U_{\alpha, n} : n < \omega\}$ of subsets of $X \setminus K_\alpha$
that are clopen in $X$. Note that then $\bigcup \{U_{\alpha, n} : n < \omega\}$ is dense open in $X$.
Next, for each $\alpha < \omega_1$ and $n < \omega$ set $V_{\alpha, n} = \bigcup \{U_{\alpha, i} : i \le n\}$.

Finally, we let $$W_\alpha = \bigcup \{\{n\} \times V_{\alpha, n} : n < \omega\},$$
clearly these are clopen sets in $\omega \times X$.
\end{definition}

It is also clear that $W_\alpha \cap (\omega\times K_\alpha) = \emptyset$ for every $\alpha < \omega_1$, hence $\bigcap_{\alpha < \omega_1} W_\alpha = \emptyset$.

\medskip

The following lemma is needed to complete the definition of the promised one-point extension of $\omega \times X$.

\begin{lemma}\label{lm:cent}
The family $\{W_\alpha : \alpha < \omega_1\}$ of clopen subsets in $\omega \times X$ is centered and has empty intersection.
\end{lemma}

\begin{proof}
We prove by induction on $n < \omega$ that if $\alpha_1 < \cdot\cdot\cdot < \alpha_n < \omega_1$ then
$$W_{\alpha_1} \cap \cdot\cdot\cdot \cap W_{\alpha_n} \ne \emptyset.$$
Indeed, if $W = W_{\alpha_1} \cap \cdot\cdot\cdot \cap W_{\alpha_{n-1}} \ne \emptyset$, then there is some $k < \omega$
with $W \cap (\{k\} \times X) \ne \emptyset$, so $W \cap (\{k\} \times X) = \{k\} \times U$ for some non-empty clopen $U$ in $X$.
Since, for any $\alpha$, the $V_{\alpha, m}$'s are increasing in $m$, this implies $W \cap (\{m\} \times X) \supset \{m\} \times U$
whenever $k < m <\omega$.

Since $\bigcup \{V_{\alpha, m} : m < \omega\}$ is dense in $X$,
it is clear from our definitions that then for any $\alpha < \omega_1$ there is some $M_\alpha < \omega$  such that
$U \cap V_{\alpha, m} \ne \emptyset$ whenever $M_\alpha \le m < \omega$. Consequently, if $m \ge \max \{k, M_{\alpha_n}\}$
then we have $W \cap W_{\alpha_n} \cap (\{m\}  \times V_{\alpha_n, m} ) \ne \emptyset$.

That $\bigcap_{\alpha < \omega_1} W_\alpha=\emptyset$, was already observed.
\end{proof}

Now, let $\mathcal{W}$ be the collection of all finite intersections of the $W_\alpha$'s.
Then $\bigcap \{W_\alpha : \alpha < \omega_1\} = \emptyset$ implies $\bigcap \mathcal{W} = \emptyset$, hence if we fix
a point $p \notin \omega \times X$ and declare the family $\{\{p\} \cup W : W \in \mathcal{W}\}$ to be a neighborhood base
for $p$ in  $L(X) = (\omega \times X) \cup \{p\}$, then $L(X)$ is indeed a 0-dimensional $T_2$-space.

Note that if the family $\{K_\alpha : \alpha < \omega_1\}$ of closed
nowhere dense subsets of $X$ is increasing, i.e. $\alpha < \beta < \omega_1$ implies $K_\alpha \subset K_\beta$, then
every countable subset $S$ of $\omega \times X$ is included in $\omega \times K_\alpha$ for some $\alpha < \omega_1$,
hence $S \cap W_\alpha = \emptyset$. So, in this case $p \notin \overline{S}$,
the closure taken in $L(X)$, and so we have $\chi(p, L(X)) = t(p, L(X)) = \omega_1$.

The idea for constructing filter bases like $\mathcal{W}$ goes back to old results on remote points, see for example 
\cite{D} and  \cite{D}.

\medskip

We are now ready to present our first result that yields some consistent counterexamples to Tkachuk's question.

\begin{theorem}\label{tm:cptL}
If there is a 0-dimensional compact $L$-space then there is a $\sigma$-compact hereditarily Lindelöf (in short: HL) 0-dimensional $T_2$-space
that is not countably tight but in which the closure of any discrete subset is
first countable.
\end{theorem}

\begin{proof}
Let $X$ be a non-separable  HL 0-dimensional compact $T_2$-space.
By a classical result of Shapirovskii from \cite{Sh1}, see also 3.13 of \cite{J}, then $d(X) = \omega_1$.
By deducting from $X$ the union of all separable open subsets, that is separable by the HL property,
we may assume that $X$ is nowhere separable, i.e. every countable set in $X$ is nowhere dense.

Now, let $\{x_\alpha : \alpha < \omega_1\}$ be a dense subset of $X$. Then $K_\alpha = \overline{\{x_\beta : \beta \le \alpha\}}$
is a non-empty closed nowhere dense set in $X$, moreover $X = \bigcup \{K_\alpha : \alpha < \omega_1\}$
because HL compact $T_2$-spaces are first countable.

Thus we have all the assumptions of Definition \ref{df:L(X)} satisfied, so we may take $L(X)$ that is now clearly $\sigma$-compact
and HL.  As the $K_\alpha$'s are increasing and every discrete subset of $\omega \times X$ is countable, we may immediately
conclude that the closure of every discrete subset of $L(X)$ is first countable, while $t(p, L(X)) = \omega_1$,
hence $L(X)$ is not countably tight.
\end{proof}

Both the ``double arrow modification" of a compact Suslin line and Kunen's compact L-space constructed under the Continuum Hypothesis in \cite{Ku}
are 0-dimensional compact $L$-spaces, hence we arrive at the following corollary of Theorem \ref{tm:cptL}.

\begin{corollary}
If there is a Suslin line or the Continuum Hypothesis holds then there is a $\sigma$-compact and HL $T_3$-space
in which the closure of every discrete subset is first countable but is not countably tight.
\end{corollary}

Using J. Moore's ZFC L-space, it was noted in \cite{VTk} that countable tightness is not discretely reflexive
in HL $T_3$-spaces. So, our above result is an, albeit only consistent, strengthening of this.

We also mention here that for an appropriate version $X$ of Moore's ZFC L-space, that is 0-dimensional, we may
conclude that $L(X)$ is a 0-dimensional $T_2$-space of weight $\omega_1$ that is not discretely generated
because $p$ is not in the closure of any discrete, hence countable, subset of $\omega \times X$.
This fact actually was noted already in \cite{DTTW}, except that when this was published the existence
of an L-space in ZFC was still not known.

Finally, we mention that a more general version of the operator $L(X)$ can be given for appropriate Tykhonov~$X$.
This would give a stronger version of Theorem \ref{tm:cptL} by omitting the assumption of 0-dimensionality.
However, the fact is that the examples of compact L-spaces that e know immediately give us the existence of
0-dimensional such spaces. So, we decided to restrict the definition of $L(X)$ for 0-dimensional $X$ because
this case is much simpler.

It had been known for long, much earlier than J. Moore found his ZFC L-space, that the
existence of an L-space implies that of a 0-dimensional one.
This lead us to raise the following question.

\begin{problem}
Does the existence of a compact L-space imply that of a 0-dimensional one?
\end{problem}

\section{ZFC examples}

For simplicity, we assume in this section that all spaces are Tykhonov.

We start by recalling the following deep
result of A. Dow in \cite{D}: If $X$ is any non-pseudocompact CCC space with $\pi(X) \le \omega_1$
then $X$ has a remote point, i.e. a point $p \in \beta X \setminus X$ that is not in the closure
of any nowhere dense subset of $X$. From this we can easily deduce the following proposition.

\begin{proposition}\label{pr:rem}
Let $X$ be a nowhere separable, non-pseudocompact,  CCC space of $\pi$-weight $\pi(X) \le \omega_1$.
Then $X$ has a one-point extension $Y = X \cup \{p\}$ such that
$p \notin \overline{A}$ for any $A \subset X$ that is countable or discrete.
\end{proposition}

\begin{proof}
By the above quoted result of A. Dow, $X$ has a remote point $p \in \beta X \setminus X$.
Clearly, then $Y = X \cup \{p\}$, taken as a subspace of $\beta X$, is as required.
(Note that any nowhere separable space is crowded, hence discrete subsets of it are nowhere dense.)
\end{proof}

Now, the $\sigma$-product $X = \{x \in 2^{\omega_1} : |x^{-1}(1)| < \omega\}$ in $2^{\omega_1}$ is
$\sigma$-compact, Fr\'echet-Urysohn, hence countably tight, moreover it satisfies all the
requirements of Proposition \ref{pr:rem}. It is non-pseudocompact, because it is non-compact.

Consequently, if $p$ is any remote point of $X$ then
$X \cup \{p\}$, taken as a subspace of $\beta X$, is a $\sigma$-compact space in which all closures
of discrete subsets are countably tight, while $X \cup \{p\}$ is not.

Thus we indeed have a rather simple ZFC example showing that countable tightness is not discretely
reflexive in the class of $\sigma$-compact spaces.

The consistent examples from the previous section had the stronger property that the closures
of discrete subsets were even first countable. It is easy to find discrete subsets in
the $\sigma$-product $X$ in $2^{\omega_1}$ that are not first countable. Our next ZFC example
does have this stronger property, but to get there we need some preparatory work.

\begin{lemma}\label{lm:nwsep}
Let $X$ be any non-separable CCC space. Then there is a regular closed subset $F$ of $X$
that is nowhere separable as a subspace.
\end{lemma}

\begin{proof}
Let $\mathcal{U}$ be a maximal disjoint family of separable open subsets of $X$. Then
$\overline{\cup \mathcal{U}}$ is separable, since $\mathcal{U}$ is countable.
But then $V = X \setminus \overline{\cup \mathcal{U}}$ is nowhere separable by the maximality
of $\mathcal{U}$, hence so is $F = \overline{V}$.
\end{proof}

\begin{lemma}\label{lm:M1}
Let $X$ be a first countable and non-separable CCC space.
Then there is a nowhere separable
closed CCC subspace $Z$ of $X$ with $d(Z) = \omega_1$.
\end{lemma}

\begin{proof}
Let us  fix for every point $x \in X$ a countable neighborhood base $\mathcal{B}_x$.
We shall say that a set $S \subset X$ is {\em $\mathcal{B}$-full} if for any $x, x' \in S$
and $B \in \mathcal{B}_x,\,\, B' \in \mathcal{B}_{x'}$, $\,\,B \cap B' \ne \emptyset$
implies $\,\,B \cap B' \cap S \ne \emptyset$.

\begin{claim}
For every countable $A \subset X$ there is a countable $\mathcal{B}$-full set $S(A)$ with $A \subset S(A)$.
\end{claim}

This can be shown by a standard closure argument but the simplest way to prove it is by taking a countable
elementary submodel $M$ of $H_\vartheta$ for a large enough regular cardinal $\vartheta$
that contains $A, X$, and the map $x \rightarrowtail \mathcal{B}_x$, and then putting $S(A) = X \cap M$.

Using this claim we then define by transfinite recursion on $\alpha < \omega_1$ an increasing
sequence $\<A_\alpha : \alpha < \omega_1\>$ of countable subsets of $X$ as follows.

We let $A_0$ be an arbitrary countably infinite subset of $X$. If $\alpha$ is limit we just put
$A_\alpha = \bigcup \{A_\beta : \beta < \alpha\}$.
If $A_\alpha$ is given then we first choose $x_\alpha \in X$ such that $x_\alpha \notin \overline{A_\alpha}$
and then let $A_{\alpha+1} = S(A_\alpha \cup \{x_\alpha\})$. (This uses that $X$ is non-separable.)

Having completed the recursion, we let $A = \bigcup \{A_\alpha : \alpha < \omega_1\}$ and $Y = \overline{A}$.
Note that we have $Y = \bigcup \{\overline{A_\alpha} : \alpha < \omega_1\}$ because $X$ is first countable.

Since $|A| = \omega_1$, we have $d(Y) \le \omega_1$. On the other hand, every countable subset $H$ of $Y$
is included in $\overline{A_\alpha}$ for some $\alpha < \omega_1$, and
then $x_\alpha \notin \overline{A_\alpha}$
shows that $H$ is not dense in $Y$. So, indeed, $d(Y) = \omega_1$.

To see that $Y$ is CCC, we consider any collection  $\{U_\alpha : \alpha < \omega_1\}$ of open subsets of
$X$ such that $U_\alpha \cap Y \ne \emptyset$ for all $\alpha < \omega_1$.
Then, as $A$ is dense in $Y$, we may pick $a_\alpha \in A \cap U_\alpha$ and then $B_\alpha \in \mathcal{B}_{a_\alpha}$
with $B_\alpha \subset U_\alpha$ for each $\alpha < \omega_1$.
 
Since $X$ is CCC, there are distinct $\alpha, \beta < \omega_1$ such that $B_\alpha \cap B_\beta \ne \emptyset$.
We may then pick $\gamma < \omega_1$ for which both $a_\alpha, a_\beta \in A_\gamma$. This, in turn,
implies $A_{\gamma+1} \cap B_\alpha \cap B_\beta \ne \emptyset$ because $A_{\gamma+1}$ is $\mathcal{B}$-full.
Consequently, we have $A \cap U_\alpha \cap U_\beta \ne \emptyset$ as well.

Finally, we may apply Lemma \ref{lm:nwsep} to obtain a regular closed subset $Z$ of $Y$ that
is nowhere separable. Clearly, then $Z$ is also CCC and $d(Z) = d(Y) = \omega_1$.
\end{proof}

We are now ready to present our final result that, in addition to A. Dow's above mentioned deep theorem,
makes use of another deep result due to M. Bell~\cite{B}.

\begin{theorem}\label{tm:Bell}
There is a first countable, nowhere separable, $\sigma$-compact and non-compact CCC space $X$
of density $\omega_1$. Then $\pi(X) = \omega_1$ as well, hence it has a remote point $p$. Consequently,
$X \cup \{p\}$, taken as a subspace of $\beta X$, is not countably tight, while
the closures of all its discrete subsets are first countable.
\end{theorem}

\begin{proof}
M. Bell constructed in \cite{B} a first countable, non-separable, $\sigma$-compact CCC space~$Y$.
By Lemma \ref{lm:M1}, $Y$ has a nowhere separable CCC closed subspace $X$ with density $d(X) = \omega_1$.
If $X$ is not compact we are done.

If, however, $X$ happens to be compact then, as $X$ is first countable, for any point $x \in X$ its
complement $X \setminus \{x\}$ is $\sigma$-compact and not compact, and inherits all the other
properties of $X$.
\end{proof}

{\em Acknowledgment.} The research on this work was done during a visit of the second author at
the HUN-REN Rényi Institute of Mathematics. He thanks the Institute for its hospitality, as well
as the Mathematics Section of the Hungarian Academy of Sciences for supporting financially
this visit.

\def\cprime{$'$}
\makeatletter \renewcommand{\@biblabel}[1]{\hfill[#1]}\makeatother

\end{document}